\documentclass[12pt]{amsart}

\usepackage{amscd,amssymb,url,colonequals,enumitem,xy, comment}
\xyoption{all}

\newtheorem{theorem}{Theorem}
\newtheorem{lemma}[theorem]{Lemma}
\newtheorem{proposition}[theorem]{Proposition}
\newtheorem{corollary}[theorem]{Corollary}

\theoremstyle{definition}

\theoremstyle{remark}
\newtheorem{remark}[theorem]{Remark}

\newcommand{\bbP}{{\mathbb P}}

\renewcommand{\P}{{\mathbb P}}
\renewcommand{\O}{\operatorname{O}}
\newcommand{\Hyp}{\operatorname{Hyp}}
\newcommand{\DHyp}{\operatorname{SDHyp}}

\newcommand{\diag}{\operatorname{diag}}

\newcommand{\tr}{{\rm tr}}

\newcommand{\Sym}{{\rm Symm}}
\newcommand{\Mat}{{\rm Mat}}

\newcommand{\GL}{{\rm GL}}

\begin{document}

\title[Symmetric determinantal hypersurfaces]{On the number of symmetric presentations of a determinantal hypersurface}
\author[Matthew Brassil and Zinovy Reichstein]{Matthew Brassil 
and Zinovy Reichstein}

\address
{Department of Mathematics \\
University of British Columbia \\
Vancouver
\\
CANADA}

\email{mbrassil@math.ubc.ca, reichst@math.ubc.ca}

\thanks{ Matthew Brassil was partially supported by a Graduate Research Fellowship from the University of British Columbia.
Zinovy Reichstein was partially supported by
National Sciences and Engineering Research Council of
Canada Discovery grant 253424-2017. }


% 14J70 (1991-now) Hypersurfaces
% 14M12 (1980-now) Determinantal varieties
% 15A22 (1980-now) Matrix pencils [See also 47A56]
% 15A86 (2010-now) Linear preserver problems

  
\begin{abstract} A hypersurface $H$ in $\bbP^r$ of degree $n$ is called determinantal if it is the zero locus of 
a polynomial of the form $\det(x_0 A_0 + \ldots + x_r A_r)$ for some $(r+1)$-tuple of $n \times n$ matrices $A = (A_0, \ldots, A_r)$.
We will refer to $A$ as a presentation of $H$. Another presentation $B = (B_0, B_1, \ldots, B_r)$ of $H$ can be obtained 
by choosing $g_1, g_2 \in \GL_n$ and setting $B_i = g_1 A_i g_2$ for every $i = 0, 1, \ldots, r$. 
In this case $A$ and $B$ are called equivalent. The second author and
A. Vistoli have shown that for $r \geqslant 3$ a general determinantal hypersurface 
admits only finitely many presentations up to equivalence. In this paper we prove a similar result for symmetric presentations for every
$r \geqslant 2$. Here the matrices $A_0, \ldots, A_r$ are required to be symmetric, and two $(r+1)$-tuples of $n \times n$ symmetric 
matrices $A = (A_0, A_1, \ldots, A_r)$ and $B = (B_0, B_1, \ldots, B_r)$ are
considered equivalent if there exists a $g \in \GL_n$ such that $B_i = g^t A_i g$ for every $i = 0, \ldots, r$. 
\end{abstract}

\maketitle

\section{Introduction}
Throughout this paper $k$ will be an algebraically closed field. We will denote the spaces of $n \times n$ matrices and symmetric $n \times n$ matrices
over $k$ by $\Mat_n$ and $\Sym_n$, respectively.
Given an $(r+1)$-tuple $A = (A_0, A_1, \ldots, A_r) \in \Mat_n^{r+1}$, we define the generalized characteristic polynomial of $A$ to be 
\[ P_A(x_0, \ldots, x_r)  = \det(x_0 A_0 + x_1 A_1 + \ldots + x_r A_r) \]
and the determinantal hypersurface $H_A$ associated to $A$ to be the hypersurface
\begin{equation} \label{e.presentation}
P_A(x_0, \ldots, x_r)  = 0 
\end{equation}
in $\bbP^r$. We will refer to the $(r+1)$-tuple $A$ as a ``presentation" of this determinantal hypersurface.
We will say that $B = (B_0, \ldots, B_r) \in \Mat_n^{r+1}$ is equivalent to $A$ if there exist
non-singular $n \times n$ matrices $g_1$ and $g_2$ such that $B_i = g_1 A_i g_2$ for every $i = 0, 1, \ldots, r$. Clearly
$H_A = H_B$ if $A$ and $B$ are equivalent but the converse is not true in general. Our staring point is the following theorem from~\cite{rv}.

\begin{theorem} \label{thm.rv} Assume $r \geqslant 3$. For
 $A \in (\Mat_n)^{r + 1}$ in general position the determinantal hypersurface $H_A$ has
only finitely many presentations up to equivalence.
\end{theorem}

% In other words, for $r \geqslant 3$, a general determinantal hypersurface has finitely many presentations, up to equivalence.
The purpose of this note is to prove an analogous assertion
for symmetric determinantal hypersurfaces. A symmetric determinantal hypersurfaces of degree $n$ in $\bbP^r$ is a hypersurface
$H = H_A$ where $A = (A_0, \ldots A_r)$ is an $(r +1)$-tuple of $n \times n$ symmetric matrices. In this case we will refer to
$A$ as a symmetric presentation of $H$.  The definition of equivalence is modified as follows.
We say that $A = (A_0, A_1, \ldots, A_r)$ and $B = (B_0, B_1, \ldots, B_r) \in \Sym_n^r$ are equivalent 
if there exists a non-singular $n \times n$-matrix $g$ such that
$B_i = g^t A_i g$ for every $i = 0, 1, \ldots, r$. Here $g^t$ denotes the transpose of $g$. This is the only type of equivalence 
we will use in the sequel.
Our main result is as follows.

\begin{theorem}\label{thm.main}
Assume $r \geqslant 2$ and the base field $k$ is of characteristic $\neq 2$. Then for
$A \in (\Sym_n)^{r + 1}$ in general position, the determinantal hypersurface $H_A$ has
only finitely many symmetric presentations up to equivalence.
\end{theorem}

%%%%%%%%%%%%%%%%%%%%%%%%%%%%%%%%%%%%%%%%%%%%%%%%
% 
% \smallskip (1) 
% The assumption that $k$ is algebraically closed is harmless. Passing to the algebraic closure, one readily sees that Theorems~\ref{thm.rv} 
% and~\ref{thm.main} remain valid for any infinite field $k$ (we still need to assume that $\cchar(k) \neq 2$ in Theorem~\ref{thm.main}). 
% If $k$ is a finite field, the notion of general position no longer makes sense, and the conclusions of Theorems~\ref{thm.rv} and~\ref{thm.main}
% become vacuous. (For a fixed $n$, both $\Mat_n(k)$ and $\Sym_n(k)$ have finitely many elements.)
% 
%%%%%%%%%%%%%%%%%%%%%%%%%%%%%%%%%%%%%%%%%%%%%%%%%%%
Note that Theorem~\ref{thm.main} cannot be deduced from Theorem~\ref{thm.rv}. Indeed, Theorem~\ref{thm.rv} concerns $(r+1)$-tuples 
$A = (A_0, \ldots, A_r)$ of $n \times n$ matrices in general position. This means that there exists a non-empty (and thus dense) Zariski open subset 
$U \subset \Mat_{n \times n}^{r + 1}$ such that Theorem~\ref{thm.rv} holds for every $A \in U$. We do not know what this open subset $U$ is; in particular,
it is possible that every $(r+1)$-tuple of symmetric matrices may be in the complement of $U$. For this reason we cannot extract any information
about presentations of a general symmetric determinantal variety $H_A$ from Theorem~\ref{thm.rv}. We also remark that Theorem~\ref{thm.rv} fails for $r = 2$ 
 (see~\cite[Remark (5), p. 615]{rv}), while Theorem~\ref{thm.main} is valid for every $r \geqslant 2$.

The remainder of this paper is structured as follows.
Sections~\ref{sect2}-\ref{sect4}
will be devoted to proving Theorem~\ref{thm.main}. In Section~\ref{sect5} we will study the number of inequivalent symmetric presentations
of a general symmetric determinantal hypersurface. This number is unknown in general; we will show that it is $1$ if $r$ is sufficiently large.

\section{Reduced tuples of symmetric matrices}
\label{sect2}

Let $U_{r, n}$ be the subset of $\Sym_n^{r+1}$ consisting of $(r+1)$-tuples $A = (A_0, \ldots, A_r)$ such that 

\smallskip
(i) $A_0$ is non-singular and 

\smallskip
(ii) the polynomial $f(t) = \det(t A_0 - A_1) = 0$ has $n$ has distinct roots. 

\smallskip \noindent
It is easy to see that $U_{r, n}$ is a non-empty (and hence, dense) Zariski open subset of $\Sym_n^{r+1}$. 

\begin{lemma} \label{lem.prel0}
(a) Suppose $H_A = H_B$ for some $A = (A_0, A_1, \ldots, A_r)$ and $B = (B_0, B_1, \ldots, B_r) \in \Sym_n^{r+1}$. Then $A$ lies in $U_{r, n}$ if and 
only if $B$ lies in $U_{r, n}$.

(b) Suppose $A$ and $B \in \Sym_n^{r+1}$ are equivalent. Then $A$ lies in $U_{r, n}$ if and only if $B$ lies in $U_{r, n}$.
\end{lemma}

\begin{proof} (a) By our assumption there exists a $0 \neq c \in k$ such that
\begin{equation} \label{e.c}
\det(x_0 A_0 + x_1 A_1 + \ldots + x_r A_r) = c \det(x_0 B_0 + x_1 B_1 + \ldots + x_r B_r) 
\end{equation}
as polynomials in $x_0, \ldots, x_r$. Setting $x_0 = 1$ and $x_1 = \ldots = x_r = 0$, we see that
then 
\[
\det(A_0) = c \det(B_0). 
\]
Similarly, setting $x_0 = t$, $x_1 = -1$ and $x_2 = \ldots = x_r = 0$, we  
conclude that
\[
\det(tA_0 - A_1) = c \det(t B_0 - B_1). 
\] 
Thus $\det(A_0) \neq 0$ if and only if $\det(B_0) \neq 0$ and $\det(tA_0 - A_1)$ has distinct roots if and only if $\det(tB_0 - B_1)$ has
distinct roots. This proves part (a). Now part (b) readily follows, because if $A$ and $B$ are equivalent, then $H_A = H_B$.
\end{proof}

We will say that an $(r+1)$-tuple $(A_0, A_1, \ldots, A_r)$ is reduced if $A_0 = I$ is the $n \times n$ identity matrix and $A_1$ is a diagonal matrix with
distinct eigenvalues. Clearly every reduced $(r+1)$-tuple lies in $U_{r, n}$.
%%%%%%%%%%%%%%%%%%%%%%%%%%%%%%%%%%%%%%%%%%%%%%%%%%%%%%%%%%%
% \begin{lemma} \label{lem.red0}
% If $A, B \in \Sym_n^r$ are both reduced, then $H_A = H_B$ if and only if $P_A = P_B$.
% \end{lemma}
%
% \begin{proof} If $P_A$ and $P_B$ are the same, then clearly their zero loci $H_A$ and $H_B$ are the same.
% Conversely, suppose $H_A = H_B$ for some $A = (A_0, A_1, \ldots, A_r)$ and $B = (B_0, B_1, \ldots, B_r)$. 
% Then $P_B = c P_A$ for some $0 \neq c \in k$, i.e.,~\eqref{e.c} holds. Setting
% $x_0 = 1$ and $x_1 = \ldots = x_r = 0$ and remembering that $A_0 = B_0 = I$,
% we see that $c = 1$.
% \end{proof}

Let $\mu_2^n \subset \GL_n$ be the subgroup of diagonal $n \times n$ matrices 
with $\pm 1$ on the diagonal, $\operatorname{S}_n \subset \GL_n$ be the subgroup 
of $n \times n$ permutation matrices and
$H_n \simeq \mu_2^n \rtimes \operatorname{S}_n$ be the subgroup of $\GL_n$
generated by $\mu_2^n$ and $\operatorname{S}_n$. It is easy to see that if $A$ is reduced, then so is
$h^t A h$ for any $h \in H_n$.

\begin{lemma} \label{lem.red} 
Let $A = (A_0, \ldots, A_r) \in U_{r, n}$. Then

\smallskip
(a) the equivalence class of $A$ contains a reduced $(r+1)$-tuple.

\smallskip
(b) Reduced $(r+1)$-tuples in $\Sym_n^{r+1}$ equivalent to $A$ are transitively permuted by $H_n$. 
In particular, there are only finitely many of them.   
\end{lemma}

\begin{proof} (a) Recall that $n \times n$ symmetric matrices $X$ are in a natural bijective correspondence 
with symmetric bilinear forms $b \colon k^n \times k^n \to k$. The symmetric bilinear form corresponding to $X$ is
 \[ b({\bf v}, {\bf w}) = {\bf v}^t X {\bf w}, \]
and the symmetric matrix corresponding to $b$ is the Gram matrix
\[ X = \begin{pmatrix} 
b({\bf e}_1, {\bf e}_1) & \hdots & b({\bf e}_1, {\bf e}_n) \\
\vdots & \hdots & \vdots \\
b({\bf e}_n, {\bf e}_1) & \hdots & b({\bf e}_n, {\bf e}_n) 
\end{pmatrix} . \]
Here ${\bf v}$ and ${\bf w}$ denote column vectors in $k^n$ and
\[ {\bf e}_1 = \begin{pmatrix} \; 1 \; \\ \;  0 \;  \\ \vdots \\ \; 0 \; \end{pmatrix} , \quad \ldots \quad , \, {\bf e}_n =
\begin{pmatrix} \; 0 \; \\ \vdots \\ \; 0 \;  \\ \; 1 \; \end{pmatrix}  \]
is the standard basis of $k^n$.
Under this correspondence, non-degenerate symmetric bilinear forms correspond to non-singular matrices and the $\GL_n$-action 
on symmetric bilinear forms by $(g \cdot b)({\bf v}, {\bf w}) = b(g({\bf v}), g({\bf w}))$  corresponds to the $\GL_n$-action on symmetric 
matrices by $g \cdot X = g^t X g$. Since $k$ is an algebraically closed field of characteristic $\neq 2$, every non-degenerate symmetric bilinear
form on $k^n$ is equivalent to 
\[ \Bigl\langle \; \begin{pmatrix} x_1 \\ \vdots \\ x_n \end{pmatrix} , \begin{pmatrix} y_1 \\ \vdots \\ y_n \end{pmatrix} \; \Bigr\rangle = x_1 y_1 + \ldots + x_n y_n. \]
In other words, there exists a $g \in \GL_n$ such that $g^t A_0 g = I$. By Lemma~\ref{lem.prel0}, since $A$ lies in $U_{r, n}$, so does $g^t A g$. 
Thus after replacing $A$ by $g^t A g$, we may assume that $A_0 = I$.

By the definition of $U_{r, n}$, $A_1$ has distinct eigenvalues $\lambda_1, \ldots, \lambda_n$. Since $A_1$ is symmetric, 
the respective eigenvectors ${\bf v}_1, \ldots, {\bf v}_n$ are mutually orthogonal relative to 
$\langle \; \; , \; \; \rangle$. Hence, there exists an orthogonal matrix $h \in \O_n$ which takes each standard basis vector ${\bf e}_i$ to a scalar
multiple of ${\bf v}_i$. 
Setting $B = (B_0, B_1, \ldots, B_r) = h^t A h$ and noting that by our choice of $h$,
$B_0 = h^t I h = I$ and $B_1 = h^t A_1 h = h^{-1} A_1 h = \diag(\lambda_1, \ldots, \lambda_n)$, we conclude that $B$ is reduced, as desired.  

(b) Suppose $B = (I, B_1, \ldots, B_r)$ and $B' = (I, B_1', \ldots, B_r')$ are reduced $(r+1)$-tuples in $\Sym_n^{r+1}$, both equivalent to $A$.
Then they are equivalent to each other. In other words, there exists a $g \in \GL_n$ such that $g^t B g = B'$. In particular,
(i) $g^t I g = I$ and (ii) $g^t B_1 g = B_1'$, where $B_1$ and $B_1'$ are both diagonal matrices with distinct eigenvalues. 

(i) tells us that $g$ is an orthogonal matrix, i.e., $g^t = g^{-1}$. Now (ii) is equivalent to saying that $g$ lies in the normalizer $N$ 
of the diagonal maximal torus $T$ in $\GL_n$. It is well known that $N = T \rtimes \operatorname{S}_n$. Since $\operatorname{S}_n$
lies in $\O_n$ and $T \cap \O_n = \mu_2^n$, we conclude that $N \cap \O_n = H_n$, and part (b) follows.
\end{proof}
% 
% \begin{remark} One can slightly strengthen Lemma~\ref{lem.red} by showing that reduced $(r+1)$-tuples form a $(\GL_n, H_n)$-section
% for the $\GL_n$-action on $\Sym_n^{r+1}$. This means that for 
% \end{remark}

\section{The locus of symmetric determinantal hypersurfaces}
\label{sect3}

Let $\Hyp_{r, n} \simeq \bbP^{\binom{r + n}{n} - 1}$ 
denote the space of degree $n$ hypersurfaces in $\bbP^r$ and 
consider the rational map 
\begin{equation} \label{e.P} \phi_{r, n} \colon \Sym_n^{r+1} \dasharrow \Hyp_{r, n}
\end{equation}
taking $A = (A_0, \ldots, A_r)$ to the hypersurface $H_A$ given by~\eqref{e.presentation}.
Let $\DHyp_{r, n}$ be the closure of the image of $\phi_{r, n}$ in $\Hyp_{r, n}$.
This is the ``locus of symmetric determinantal hypersurfaces" of degree $n$ in $\bbP^r$.

\begin{proposition} \label{prop.prel}
Let $r,n$ be positive integers and $k$ be a field of characteristic $\neq 2$. Then the following assertions are equivalent.

\smallskip
(a) Given an $(r+1)$-tuple $A=(A_0, A_1,\dots,A_r)\in \Sym_n^{r+1}$ in general position, there are
only finitely many $(r+1)$-tuples $B \in \Sym_n^{r + 1}$ (up to equivalence) such that $H_A = H_B$.

\smallskip
(b) Given an $(r+1)$-tuple $A = (A_0, A_1, \ldots, A_r) \in \Sym_n^{r+1}$ in general position, there are only finitely 
many reduced $(r+1)$-tuples $B = (I, B_1, \ldots, B_r)$ such that $H_A = H_B$.

\smallskip
(c) $\dim(\DHyp_{r, n}) = (r - 1) \dfrac{n(n+1)}{2} + n$.
\end{proposition}

\begin{proof} 
(a) $\Longleftrightarrow$ (b): We may assume without loss of generality that $A \in U_{r, n}$. Fix $A$ and let $B \in \Sym_n^{r+1}$ vary over the symmetric
presentations of $H_A$. Every such $B$ lies in $U_{r, n}$ by Lemma~\ref{lem.prel0} and thus
is equivalent to at least one and at most finitely many reduced tuples by Lemma~\ref{lem.red}. 
In summary, for a given $A \in U_{r, n}$, there are finitely many presentations $B \in \Sym_n^{r+1}$ of $H_A$ up to equivalence 
if and only if there are finitely many reduced symmetric presentations.  

\smallskip
(b) $\Longleftrightarrow$ (c): Let 
$
Z_{r, n} = \{ I \} \times D_n \times \Sym_n^{r-1}
$
be the subvariety of reduced tuples in $\Sym_n^{r+1}$.
Here $D_n$ is the space of diagonal $n \times n$ matrices with distinct eigenvalues.
The rational map~\eqref{e.P} restricts to a morphism
\begin{equation} \label{e.restricted}
\phi_{r, n} \colon Z_{r, n} \longrightarrow \Hyp_{r, n} . 
\end{equation}
Note that $\phi_{r, n}(A) = \phi_{r, n}(B)$ whenever $A$ and $B$ are equivalent. By Lemma~\ref{lem.red}, up to equivalence every 
$A \in U_{r, n}$ can be represented by a $B \in Z_{r, n}$. Thus
\begin{equation} \label{e.image}
\DHyp_{r, n} = \overline{\phi_{r, n} (\Sym_n^{r+1})} = \overline{\phi_{r, n}(U_{r, n})} = \overline{\phi_{r, n} (Z_{r, n})}. 
\end{equation}
Here the horizontal bar denotes Zariski closure in $\Hyp_{r, n}$.

To finish the proof observe that (b) is equivalent to the general fiber of~\eqref{e.restricted} being finite. 
On the other hand in view of~\eqref{e.image}, (c) is equivalent to 
the image of~\eqref{e.restricted} being of dimension $(r - 1)\dfrac{n(n+1)}{2} + n$, i.e., of the same dimension as
$Z_{r, n}$. The equivalence between (b) and (c) now follows from the Fiber Dimension 
Theorem; see, e.g.,~\cite[Section I.6.3]{shafarevich}.
\end{proof}

\section{Proof of Theorem~\ref{thm.main}}
\label{sect4}

We will first prove the theorem for $r=2$. Our strategy here is to compute the dimension of $\DHyp_{2, n}$ and appeal to Proposition~\ref{prop.prel}.
It is well know that every smooth curve of degree $n \geqslant 1$ in $\bbP^2$ can be written in the form
$P_A(x_0, x_1, x_2) = 0$ for some triple $A = (A_0, A_1, A_2)$ of symmetric $n \times n$ matrices. This was first proved in~\cite{acdixon}; for a modern proof
and further references
see~\cite[Proposition 2(a)]{beauville} or~\cite[Theorem 1]{psv}. Thus $\DHyp_{2, n}$ contains a dense open subset of $\Hyp_{2, n} \simeq \bbP^{\frac{(n+2)(n+1)}{2} - 1}$.
We conclude that 
\[ \dim(\DHyp_{2, n}) = \frac{(n + 2)(n + 1)}{2} - 1 = \frac{n(n+1)}{2} + n . \]
This means that condition (c) of Proposition~\ref{prop.prel} holds for $r = 2$, and hence, so does condition (a). 

This completes the proof of Theorem~\ref{thm.main} for $r = 2$.
In view of Proposition~\ref{prop.prel} this case can be restated as follows. There exists a dense open subset $V$ of $\Sym_n^{3}$ such that
for any $(A_0, A_1, A_2) \in V$, there exist only finitely many reduced triples $(I, B_1, B_2) \in Z_{2, n}$ satisfying 
\begin{equation} \label{e.det3}  H_{(A_0, A_1, A_2)} = H_{(I, B_1, B_2)}.
\end{equation}
as polynomials in $x_0, x_1, x_2$.

Now suppose $r \geqslant 3$. We may assume without loss of generality that
\[ \text{$A = (A_0, A_1, \ldots, A_r) \in U_{r, n}$ and 
$(A_0, A_1, A_i) \in V$ for every $i = 2, 3, \ldots, r$.} \]
By Proposition~\ref{prop.prel} it suffices to show that
there are only finitely many reduced $(r+1)$-tuples
$B = (I, B_1, \ldots, B_r) \in Z_{r, n}$ such that $H_A = H_B$. In other words,
$\det(x_0 I + x_1 B_1  + \ldots + x_r B_r)$ is a non-zero scalar multiple of $\det(x_0 A_0 + x_1 A_1 + \ldots + x_r A_r)$.
Setting $x_3 = \ldots = x_r = 0$, we see that $\det(x_0 I + x_1B_1 + x_2 B_2)$ is a non-zero scalar multiple
of $\det(x_0 A_0 + x_1 A_1 + x_2 A_2)$, i.e.,~\eqref{e.det3} holds.
Remembering that $(A_0, A_1, A_2) \in V$, we conclude that 
for a fixed $A$ there are only finitely many possibilities for the reduced triple $(I, B_1, B_2)$. 
Similarly, for any $i = 3, 4, \ldots, r$,
since $(A_0, A_1, A_i) \in V$, there are only finitely many possibilities for the reduced triple $(I, B_1, B_i)$. In summary,
as $B$ ranges over the reduced symmetric presentations of $H_A$, there are only finitely many
possibilities for $B_1, B_2, \ldots, B_r$. This tells us that $H_A$ has only finitely many reduced symmetric presentations.
Theorem~\ref{thm.main} now follows from the equivalence of~(a) and (b) in Proposition~\ref{prop.prel}.
\qed

\begin{corollary} \label{cor.dimension} The dimension of the locus of symmetric determinantal varieties $\DHyp_{r, n}$ is $(r - 1)\dfrac{n(n+1)}{2} + n$
for any $r \geqslant 2$ and $n \geqslant 1$.
\end{corollary}

\begin{proof} This follows from Proposition~\ref{prop.prel}. We have proved that condition (a) holds, and hence so does condition (c). 
\end{proof}

\section{The number of symmetric presentations}
\label{sect5}

Theorem~\ref{thm.main} asserts that the number of inequivalent symmetric presentations of a general element of $\DHyp_{r, n}$ is finite.
Denote this number by $\mu(r, n)$. It is clear from the definition that $\mu(r, 1) = 1$ for any $r \geqslant 1$. It is also known that
\[ \mu(2, n) = \begin{cases} 
2^{\frac{(n-1)(n-2)}{2}} (2^{\frac{(n-1)(n-2)}{2}} + 1) - 1, \; \; \text{if $n \geqslant 11$ and $n \equiv \pm 3 \pmod{8}$} \\
2^{\frac{(n-1)(n-2)}{2}} (2^{\frac{(n-1)(n-2)}{2}} + 1), \; \; \text{otherwise};
\end{cases}
\]
see \cite[Theorem 1]{psv}. In this sections we will prove the following

\begin{proposition}\label{prop.mu}
% (a) $\mu(r, n) \geqslant \mu(r+1, n)$ for any $r, n \geqslant 2$.
%
% (b) 
If $r \geqslant \dfrac{n(n+1)}{2} - 1$, then $\mu(r, n) = 1$ for every $n \geqslant 2$. 
\end{proposition}

\begin{proof}
First assume $r = \dfrac{n(n+1)}{2} - 1$.  Choose $A = (A_0, A_1, \ldots, A_r) \in \Sym_n^{r+1}$ such that
$A_0, \ldots, A_r$ are linearly independent over $k$, i.e., $A_0, \ldots, A_r$ form a basis for $\Sym_n$ as a $k$-vector space. 
Suppose $H_A = H_B$ for some $B = (B_0, B_1, \ldots, B_r) \in \Sym_n$. This means that $P_A$ is a non-zero scalar multiple of $P_B$. 
After replacing $B$ by $g^t B g$ for a suitably chosen scalar matrix $g$, we may assume that $P_A = P_B$.
Our goal is to show that $B$ is equivalent to $A$.

Let $T:\Sym_n\rightarrow \Sym_n$ be the linear transformation sending $A_i$ to $B_i$ for each $1\leqslant i \leqslant r$. 
The condition that $P_B=P_A$ translates to $\det(T(X))=\det(X)$ for every symmetric matrix $X \in \Sym_n$. Any such linear transformation
is of the form $T(A)= g^t A g$; see~\cite[Corollary 8.6]{bgl} or~\cite[Theorem 1.1]{ct}\footnote{\cite[Corollary 8.6]{bgl} 
and \cite[Theorem 1.1]{ct} asserts that $T(A)=\alpha g^t Ag$
for some $\alpha \in k$ and $g \in\GL_n$ such that $\det(\alpha g^2)=1$. 
However, since we are assuming that our base field $k$ is algebraically closed,
we can absorb $\alpha$ by into $g$ by replacing $g$ with $\dfrac{1}{\sqrt{\alpha}} g$.}.
This shows that $A$ and $B$ are equivalent, as desired.

We will show that $\mu(r, n) = 1$ for any $r \geqslant \dfrac{n(n+1)}{2} - 1$ by induction on $r$. The base case,
where $r = \dfrac{n(n+1)}{2} - 1$, was settled in the previous paragraph. 
For the induction step, let us assume that $\mu(r, n) = 1$ for some $r \geqslant \dfrac{n(n+1)}{2} - 1$. Our goal is to
show that $\mu(r+1, n) = 1$. We will argue by contradiction. Assume that $\mu(r + 1, n) \geqslant 2$.
Then for $A = (A_0, \ldots, A_r, A_{r+1})$ in general position in $\Sym_n^{r+2}$, there exists a $B = (B_0, \ldots, B_r, B_{r+1}) \in \Sym_n^{r+2}$ 
such that $B$ is not equivalent to $A$ and $H_A = H_B$. Once again, after replacing $B$ by a scalar multiple,
we may assume that $P_A = P_B$. (Since $k$ is an algebraically closed field, any scalar multiple of $B$ is equivalent to $B$.)
In other words, 
\begin{equation} \label{e.mu} \det(x_0 A_0 + \ldots + x_r A_r + x_{r+1} A_{r + 1}) =  
\det (x_0 B_0 + \ldots + x_r B_r + x_{r+1} B_{r+1}). 
\end{equation}
Setting $x_{r+1} = 0$, we see that
$P_{\overline{A}} = P_{\overline{B}}$, and thus $H_{\overline{A}} = H_{\overline{B}}$,
where \[ \text{$\overline{A} = (A_0, A_1, \ldots, A_r)$ and $\overline{B} = (B_0, B_1, \ldots, B_r)$.} \]
Since $A$ is in general position in $\Sym_n^{r+2}$, $\overline{A}$ is in general position on $\Sym_n^{r+1}$. By our assumption, $\mu(r, n) = 1$; hence,
$\overline{A}$ and $\overline{B}$ are equivalent, i.e., $\overline{A} = g^t \overline{B} g$ for some $g \in \GL_n$.
After replacing $B$ by $g^t B g$, we may assume that $B_0 = A_0, \ldots, B_r = A_r$. In other words,
$A = (A_0, \ldots, A_r, A_{r+1})$, $B = (A_0, \ldots, A_r, B_{r + 1})$ and $B$ is not equivalent to $A$. 

We claim that this is not possible. More specifically, we claim that in this situation our assumption that $P_A = P_B$ forces $B_{r+1}$ to be equal 
to $A_{r+1}$. Since $r \geqslant \dfrac{n(n+1)}{2} - 1$ and $A$ is in general position in $\Sym_n^{r+1}$, we may assume without loss of generality that 
$A_0, \ldots, A_r$ span $\Sym_n$ as a $k$-vector space. 

Now recall that the symmetric bilinear form $(X, Y) \mapsto \tr(XY)$ (otherwise known as the trace form) is non-degenerate on 
$\Sym_n$. Thus in order to prove the claim, it suffices to show that
\begin{equation} \label{e.C} \text{$\tr(C A_{r+1}) = \tr(C B_{r+1})$ for every $C \in \Sym_n$.}
\end{equation}
Since non-singular matrices are Zariski dense in $\Sym_n$, we only need to prove~\eqref{e.C} for every non-singular symmetric matrix $C \in \Sym_n$.
To establish~\eqref{e.C} for a non-singular matrix 
$C \in \Sym_n$, note that $C^{-1}$ is also symmetric and thus can be written as a linear combination 
\[ C^{-1} = c_0 A_0 + \ldots + c_r A_r \]
for suitable scalars $c_0, c_1, \ldots, c_r \in k$.
Substituting $x_i = \lambda c_i$ for $i = 0, \ldots, r$ and $x_{r+1} = -1$ into~\eqref{e.mu}, where $\lambda$ is a variable, 
and remembering that $A_i = B_i$ for $i = 0, 1, \ldots, r$, we see that
\[ \det(\lambda C^{-1}  - A_{r + 1}) = \det (\lambda C^{-1} -  B_{r+1}).  \]
Multiplying both sides by $\det(C)$, we conclude that $CA_{r+1}$ and $C B_{r+1}$ have the same characteristic polynomial. Hence, they have the same trace,
$\tr(CA_{r+1}) = \tr(C B_{r+1})$. This completes the proof of the claim and thus of Proposition~\ref{prop.mu}.
\end{proof}

\begin{remark} Using Lemma~\ref{lem.red}(b), one can show that 
\[ \mu(r, n) =   \frac{\deg(\phi_{r, n})}{2^{n-1} n!}, \]
where $\deg(\phi_{r, n})$ denotes the separable degree of the morphism $\phi_{r, n} \colon Z_{r, n} \longrightarrow \DHyp_{r, n}$ in~\eqref{e.restricted}.
The denominator $2^{n-1} n!$ is half the size of the group $H_n$. The reason for the ``half" is that the kernel of the $H_n$-action on $Z_{r, n}$
is the subgroup $\{ \pm I \}$ of order $2$. Thus for general $A \in \Sym_n^{r+1}$ the number of reduced $(r+1)$-tuples equivalent to $A$ is
$|H_n|/2$.
\end{remark}

\begin{remark} The description of the linear transformation $T \colon \Sym_n \to \Sym_n$ preserving the determinant which was 
used in the proof of Proposition~\ref{prop.mu} is a symmetric version of a classical 
theorem of Frobenius~\cite{frobenius}. For a detailed discussion of this and related ``preserver problems" 
we refer the reader to~\cite{bgl}.
\end{remark}

\begin{remark} We do not know what the value of $\mu(r, n)$ is for $3 \leqslant r \leqslant \dfrac{n(n+1)}{2} - 2$.
\end{remark}

\begin{remark}  A variant of Theorems~\ref{thm.rv} and~\ref{thm.main} for ``skew-symmetric" matrices
is also of interest. Here ``determinant" should be replaced by ``Pfaffian". We leave it as an open problem for the reader.
\end{remark}

\section*{Acknowledgments} We are grateful to Boris Reichstein whose questions led to both~\cite{rv} and the present paper and to
Giorgio Ottaviani for his interest in Theorem~\ref{thm.main} and helpful comments.

\bibliographystyle{abbrv}
% \bibliography{symmetric}

\end{document}